\documentclass{amsart}
\usepackage{amsmath, amsthm}
\usepackage{amssymb}
\usepackage{amsfonts}
\usepackage{latexsym}
\usepackage{mathrsfs}
\usepackage{bm}
\usepackage{hyperref}
\usepackage{graphicx}
\usepackage[dvipsnames,usenames]{color}
\hypersetup{
   unicode=false,          % non-Latin characters in AcrobatÃ¢â•¢Ë˚s bookmarks
   pdftoolbar=true,        % show AcrobatÃ¢â•¢Ë˚s toolbar?
   pdfmenubar=true,        % show AcrobatÃ¢â•¢Ë˚s menu?
   pdffitwindow=false,     % window fit to page when opened
   pdfstartview={FitH},    % fits the width of the page to the window
   pdftitle={My title},    % title
   pdfauthor={Author},     % author
   pdfsubject={Subject},   % subject of the document
   pdfcreator={Creator},   % creator of the document
   pdfproducer={Producer}, % producer of the document
   pdfkeywords={keywords}, % list of keywords
   pdfnewwindow=true,      % links in new window
   colorlinks=true,       % false: boxed links; true: colored links
   linkcolor=blue,          % color of internal links
   citecolor=blue,        % color of links to bibliography
   filecolor=magenta,      % color of file links
   urlcolor=MidnightBlue          % color of external links
}

\theoremstyle{plain}
\newtheorem{theorem}{Theorem}[section]

\newtheorem{lem}[theorem]{Lemma}
\newtheorem{pro}[theorem]{Proposition}
\newtheorem{Def}[theorem]{Definition}
\newtheorem{rem}[theorem]{Remark}
\numberwithin{equation}{section}

\newcommand{\cs}{conformal structure}

\newcommand{\cmc}{constant mean curvature}

\newcommand{\hcd}{holomorphic cubic differential}
\newcommand{\hqd}{holomorphic quadratic differential}

\newcommand{\hym}{hyperbolic metric}
\newcommand{\htm}{hyperbolic three-manifold}
\newcommand{\ift}{implicit function theorem}

\newcommand{\im}{induced metric}

\newcommand{\KE}{K\"{a}hler-Einstein}
\newcommand{\km}{K\"{a}hler manifold}

\newcommand{\lo}{linearized operator}
\newcommand{\mc}{mean curvature}
\newcommand{\mcf}{mean curvature flow}

\newcommand{\maxp}{maximum principle}
\newcommand{\mi}{minimal immersion}
\newcommand{\ml}{minimal Lagrangian}
\newcommand{\MS}{moduli space}
\newcommand{\ms}{minimal surface}
\newcommand{\mli}{minimal Lagrangian immersion}
\newcommand{\mls}{minimal Lagrangian surface}

\newcommand{\mpa}{mountain pass}

\newcommand{\mpt}{mountain pass theorem}

\newcommand{\pc}{principal curvature}
\newcommand{\qf}{quasi-Fuchsian}
\newcommand{\RS}{Riemann surface}
\newcommand{\sff}{second fundamental form}

\newcommand{\scv}{solution curve}

\newcommand{\TS}{Teichm\"{u}ller space}
\newcommand{\Tt}{Teichm\"{u}ller theory}

\newcommand{\tg}{totally geodesic}

\newcommand{\uc}{universal cover}

\newcommand{\WP}{Weil-Petersson}
\newcommand{\WPm}{Weil-Petersson metric}
\newcommand{\wrt}{with respect to}

\newcommand{\be}{\begin{equation}}
\newcommand{\ene}{\end{equation}}
\newcommand{\br}{\begin{rem}}
\newcommand{\er}{\end{rem}}
\newcommand{\bl}{\begin{lem}}
\newcommand{\el}{\end{lem}}
\newcommand{\bd}{\begin{Def}}
\newcommand{\ed}{\end{Def}}
\newcommand{\ben}{\begin{enumerate}}
\newcommand{\een}{\end{enumerate}}
\newcommand{\bp}{\begin{proof}}
\newcommand{\ep}{\end{proof}}
\newcommand{\bpo}{\begin{pro}}
\newcommand{\epo}{\end{pro}}
\newcommand{\beq}{\begin{equation*}}
\newcommand{\eeq}{\end{equation*}}
\newcommand{\bear}{\begin{eqnarray*}}
\newcommand{\eear}{\end{eqnarray*}}
\newcommand{\bt}{\begin{theorem}}
\newcommand{\et}{\end{theorem}}

\newcommand{\R}{\mathbb{R}}

\newcommand{\RH}{\mathbb{RH}^2}
\newcommand{\CH}{\mathbb{CH}^2}

\newcommand{\Acal}{\mathcal{A}}
\newcommand{\Fcal}{\mathcal{F}}

\newcommand{\Tcal}{\mathcal{T}}

\newcommand{\ddl}[2]{\frac{d{#1}}{d{#2}}}

\newcommand{\ppl}[2]{\frac{\partial{#1}}{\partial{#2}}}

\newcommand{\II}{\mathrm{I\!I}}

\numberwithin{equation}{section}

\allowdisplaybreaks

\def\XXint#1#2#3{{\setbox0=\hbox{$#1{#2#3}{\int}$}
    \vcenter{\hbox{$#2#3$}}\kern-.5\wd0}}

\makeatletter
\def\@citestyle{\m@th\upshape\mdseries}
\def\citeform#1{{\bfseries#1}}
\def\@cite#1#2{{%
  \@citestyle[\citeform{#1}\if@tempswa, #2\fi]}}
\@ifundefined{cite }{%
  \expandafter\let\csname cite \endcsname\cite
  \edef\cite{\@nx\protect\@xp\@nx\csname cite \endcsname}%
}{}
\makeatother
\begin{document}

\title
{Holomorphic cubic differentials and minimal Lagrangian surfaces in $\CH$}

\author{Zheng Huang}
\address{Department of Mathematics,
The City University of New York,
Staten Island, NY 10314, USA.}
\email{zheng.huang@csi.cuny.edu}

\author{John Loftin}
\address{Department of Mathematics and Computer Science,
Rutgers University - Newark,
Newark, NJ 07102, USA.}
\email{loftin@rutgers.edu}

\author{Marcello Lucia}
\address{Department of Mathematics,
The City University of New York,
Staten Island, NY 10314, USA.}
\email{mlucia@math.csi.cuny.edu}

\date{April 17, 2011}


\begin{abstract}
Minimal Lagrangian submanifolds of a K\"ahler manifold represent a very interesting class of
submanifolds as they are Lagrangian with respect to the symplectic structure of the ambient space,
while minimal {\wrt} the Riemannian structure. In this paper we study {\mli}s of the {\uc} of closed
surfaces (of genus $g \ge 2$) in $\CH$, with prescribed data $(\sigma, tq)$, where $\sigma$ is a
{\cs} on the surface $S$, and $q dz^3$ is a {\hcd} on the {\RS} $(S,\sigma)$. We show existence
and non-uniqueness of such {\mli}s. We analyze the asymptotic behaviors for
such immersions, and establish the surface area {\wrt} the {\im} as a {\WP} potential function for
the space of {\hcd}s on $(S,\sigma)$.
\end{abstract}

\maketitle

\section {Introduction}

The theory of minimal hypersurfaces (co-dimension one) in a Riemannian manifold has been a field of
both extraordinary depth and far-reaching width in mathematics. The situation of higher co-dimensional
minimal submanifolds can be much more complicated. In this paper, we aim to investigate
some minimal submanifolds of co-dimension two, motivated from mirror symmetry and ``Lagrangian
Plateau problem" (\cite{SW99, SW01}). The {\ml} submanifolds in various ambient spaces can be studied
as a constrained variational problem (see for instance \cite{Oh90, MW93}), and there are many interesting
analogs to the classical {\ms} theory in Riemannian manifolds. In general, there are obstructions to the existence
of {\ml} submanifolds in a Riemannian manifold, even in the case of a {\km} (\cite{Bry87b}).

For a {\km} $M^{2n}$, one studies its {\ml} submanifolds: Lagrangian {\wrt} the symplectic structure and minimal
{\wrt} the Riemannian structure of $M^{2n}$. The obstructions in \cite{Bry87b} for existence do not occur if $M^{2n}$
is a {\KE} manifold, but the general existence is still largely unknown. In this paper we consider the existence and
multiplicity of {\mli}s of the {\uc} of closed {\RS}s into the complex hyperbolic plane $\CH$, with prescribed data on
the closed surface. Each such immersion is equivariant {\wrt} an induced representation of the fundamental group of the
surface into $SU(2,1)$.

%; in this setting, there is a unique {\hs} on any {\cs}
%of closed oriented surface of genus $g \geq 2$ via the {\ut}, and
%the deformation of these {\hs}s is described by {\Tt}.
%On the other hand, the ambient space $\CH$ carries its own {\hs}.
%Our main equation can be interpreted as the
%interaction between these {\hs}s; We use $\Tcal_g \times C(\sigma)$
%to parametrize a class of {\mli}s and prove
%results on global existence and explore the uniqueness issue,
%where $\Tcal_g$ is {\TS} of $\Sigma$ and $C(\sigma)$
%is the space of {\hcd}s on the {\cs} $\sigma$ on $\Sigma$. This
%parametrization is also deeply related to the space of
%convex $\RP$ structures on the surface (\cite{Lof01, Lab07}).

It is well-known that (for instance, \cite{LJ70}), the second fundamental form of a {\ms} in a three-dimensional space
form is described as the real part of a {\hqd}. An analogous fact is true for minimal Lagrangian surfaces in $\CH$,
namely, such an immersion can be constructed from a {\cs} on a closed surface and a {\hcd} on this {\cs}. The space of
{\hcd}s on closed {\RS}s is deeply related to the space of convex flat projective structures on the surface
(\cite{Lof01, Lab07}).

Our perspective is to develop a moduli theory for {\mli}s of (the covering spaces of) closed
surfaces into $\CH$. We are particularly interested in the general existence and uniqueness properties of these
immersions for prescribed {\cs} and {\hcd}. Our method of study relies on reducing the immersion problem to the
solvability of the following equation from
\cite{LM10}:
 \be\label{e} \Delta u + 2 - 2e^u -16t^2\|q\|^2 e^{-2u} = 0,
 \ene
on a compact {\RS} equipped with a background {\hym} $g_\sigma$, {\hcd} $q$, and real positive parameter $t$. This
equation is the integrability condition of a {\ml} surface in $\mathbb{CH}^2$ with induced metric $e^ug_\sigma$ and
the {\sff} determined by $tq$.

More specifically, given the pair $(\sigma, tq)$, a solution to equation \eqref{e} gives rise to a Legendrian frame
(see \S 2.2) from the {\uc} $\tilde{\Sigma}$ to $SU(2,1)$, for a {\mli} $\varphi$ from
$\tilde{\Sigma}$ to $\CH$. Note that the group of interest $SU(2,1)$ is the triple covering of $PU(2,1)$, the
holomorphic isometry
group of $\CH$. Hence we obtain a natural representation of the fundamental group of $\Sigma$ into
$SU(2,1)$,  for which the {\mli} $\varphi$ is equivariant. The perspective of surface group representation theory is
explored in more detail in the paper \cite{LM10}.

Note that the {\im} provides a {\cs} and a background metric of constant curvature via the uniformization theorem.
Since the deformation of the {\cs}s on a closed surface is described by {\Tt}, we find extra tools to this problem,
as well as applications to {\Tt}.

Let us fix some notation and basic assumptions that will be frequently used throughout the paper.
\ben
\item
Let $\Sigma$ be a smooth, closed, oriented surface of genus $g \geq 2$, and $\sigma$ be a {\cs} on $\Sigma$,
with conformal coordinates $z$. Note that $\sigma$ is a point on {\TS} $\Tcal_g$ of {\RS}s (here we do not distinguish {\TS} and {\MS});
\item
Let $g_\sigma dzd\bar{z}$ be the {\hym} (of constant curvature $-1$) on $(\Sigma,\sigma)$, and $\Delta$ is the
Laplace operator for $g_\sigma$;
\item
Let $C(\sigma)$ be the space of {\hcd}s of the form $q(z)dz^3 $ on $(\Sigma,\sigma)$, where $\ppl{q(z)}{\bar{z}} = 0$.
Note that by the Riemann-Roch theorem, the complex dimension of $C(\sigma)$ is $5g-5$ (see for example \cite{FK80});
\item
We assign the following notation for a {\hcd} $q(z)dz^3$ with respect to the {\hym} $g_\sigma$:
\be\label{q}
\|q\| = \frac{|q|}{g_\sigma^{3/2}}.
\ene
Therefore it generates a natural $L^2$-pairing (the {\WP} pairing) of {\hcd}s in $C(\sigma)$. We always
assume $q \not\equiv 0$, but keep in mind that the cubic differential $q$ must necessarily have (finitely many) zeros on $\Sigma$.
\een
We are interested in understanding {\mls}s in $\CH$, in particular, the existence, uniqueness and asymptotic behaviors
of such immersions. Our main result can be summarized into the following:
 \bt\label{main}
Let $\Sigma$ be a closed marked surface of genus $g \ge 2$, $\sigma \in \Tcal_g$ be a {\cs} on
$\Sigma$, and $q dz^3 \in C(\sigma)$ be a {\hcd} on this marked
surface $(\Sigma,\sigma)$, then we have the following:
 \ben
\item There is a $T_0=T_0(\sigma,q)>0$ so that for any
$t\in(0,T_0)$, there are at least two immersed {\mli}s from $\tilde{\Sigma}$, the {\uc} of $\Sigma$, into $\CH$, determined by $(\sigma,tq)$;
\item There is a {\mli} from $\tilde{\Sigma}$ into $\CH$ determined by $(\sigma,T_0q)$;
\item (\cite{LM10})
There exists a $T = T(\sigma,q) > 0$ such that for any $t > T$,
there is no {\mli} of $\tilde{\Sigma}$ into $\CH$ determined by $(\sigma,tq)$.
 \een \et

%Note that the {\ift} only guarantees the existence of one solution
%on some open interval $(0, \epsilon)$ in part (ii). The inclusion of
%the endpoint $T_0$ in part (iii) is very special to this setting,
%and it provides important information on the asymptotes of solutions
%along the {\scv} for $t$ approaching $T_0$.

The most technical parts of Theorem \ref{main} are the parts (i)
and (ii). Our approach for part (i) consists of two steps: We first
(see Theorem \ref{curve}) deploy the continuity method to produce a
{\scv} to equation \eqref{e}, and show corresponding solutions
are stable, then we produce an additional solution by the {\mpt} for
each stable solution obtained in Theorem \ref{curve}. Part (ii)
essentially determines the asymptotic behavior of the {\scv} on
which the linearized operator is positive. Proving part (ii)
requires the closedness estimate in Theorem \ref{u-bound}, i.e., the
continuity method extends to the endpoint $T_0$. This estimate
relies on the compactness of the surface $\Sigma$.  We also note an
antecedent to part (i) is proved already in \cite{LM10}: There is a
$\tilde T_0\in(0,T_0]$ so that there is a single solution of (\ref{e}) for each
$t<\tilde T_0$ is Theorem 5.1 of \cite{LM10}.

As an application to {\Tt}, we show that
\bt
The induced surface area $($for a unique {\mli} corresponding to data $(\sigma,q dz^3))$ is a potential function of the
{\WP} norm in the space of {\hcd}s on $\sigma \in \Tcal_g$.
\et

Understanding Lagrangian surfaces in $\CH$ is an important ingredient in studying representations in the complex
hyperbolic {\qf} space (see \cite{PP06}). Note that the pair $(\sigma, q)$ provides a parameter space of real dimension
$16g-16$, which agrees with the real dimension of complex hyperbolic {\qf} space. Equation \eqref{e} is one of several equations
corresponding to immersing closed surface (or the {\uc}) into other geometries. It is of great interest in higher {\Tt} to
understand the space of surface group representations into higher rank Lie groups, and to integrate techniques of nonlinear analysis
with the representation theory.

It is also worth mentioning that the Lagrangian property is
preserved under the {\mcf} in {\km}s (\cite{Smo96}), while {\ms}s or
surfaces of {\cmc} are often natural candidates for limiting
submanifolds (if exist) of various {\mcf}s. The understanding of
existence and uniqueness of {\mli}s can provide important geometric
insight for the analysis of Lagrangian {\mcf}s in {\km}s (see for
instance \cite{Wan08}). Minimal Lagrangian submanifolds also play
vital roles in the geometry of calibrated submanifolds in Calabi-Yau
manifolds (\cite{HL82}), the SYZ conjecture in mirror symmetry
(\cite{SYZ96, LYZ05}), and symplectic topology
(\cite{Joy05}), to name a few.

%------------------------------------------------------
\subsection*{Plan of the paper}
This paper is organized as follows: In Section \S 2, after recalling the preliminaries of {\ml} submanifolds in $\CH$, we
set up the structure equation and reduce the {\mli} problem to the solutions to  equation \eqref{e} (Proposition \ref{reduction}),
and we relate the {\sff} of the {\mli} to the prescribed {\cs} and {\hcd} (Proposition \ref{sff}).  We prove part (iii) of Theorem
\ref{main} in the subsection \S3.1, and prove the existence of a {\scv} $\gamma$ in the subsection \S 3.2. In \S 3.3 we derive a
uniform estimate for the solutions on the {\scv} away from zero, and hence show the right endpoint $T_0$ is in fact included
on the {\scv} $\gamma$. We then complete the proof of Theorem \ref{main} in Section \S 4, where we focus on
the
non-uniqueness of {\mli}s with prescribed data. In Section \S 5, since the solutions near the trivial solution $\gamma(0)$ are
unique, we are able to define a functional on a subspace of the space of {\mli}s in $\CH$. As an application to {\Tt}, we show
this functional is a potential function of the {\WP} norm of {\hcd}s.
%------------------------------------------------------
\subsection*{Acknowledgements}
The research of Huang was supported (in part) by a grant from a PSC-CUNY Research Award Program and a CIRG-CUNY
program, the research of Loftin is supported in part by a Simons Collaboration Grant for Mathematicians 210124, and the
research of Lucia is supported by projects MTM2008-06349-C03-01, MTM2011-27739-C04-01 (Spain) and 2009SGR345
(Catalunya), and a Simons Foundation Collaboration Grant for Mathematicians 210368.
%%%%%%%%%%%%%%%%%%%%%%%%%%%%%%%%%%%%%%%
\section{Minimal Lagrangian submanifolds in $\CH$}
%This immersion problem, of fixed {\cs} $\sigma$ and prescribed
%{\hcd} $q(z)dz^3$, is described by the following semilinear
%elliptic different equation: \be\label{e} \Delta u + 2 -4e^u
%-4\|q\|^2 e^{-2u} = 0, \ene where $\Delta$ is the hyperbolic
%Laplacian, and the solution $u(z)$ is the {\cc} factor such that $g
%= e^u g_\sigma$ is a smooth metric in the {\cs} $\sigma$.
\subsection{Complex Hyperbolic Space}
Before we move to our main interest in $\CH$, let us briefly mention a few general facts on {\ml} submanifolds in a {\km}.
Let $(M^{2n},\omega)$ be a {\km} where $\omega$ is its K\"{a}hler form. Let $N^n$ be a submanifold of dimension $n$ in
$M^{2n}$. The inclusion map $i: N^n \to M^{2n}$ is called {\it Lagrangian} if
\beq
i^{*} \omega \equiv 0.
\eeq
In other words, a Lagrangian submanifold is characterized by the vanishing of $\omega|_N$. In terms of the
Riemannian structure on $M^{2n}$, the submanifold $N^n$ is Lagrangian if the tangent space $TM$ restricted on
$N^n$ is an orthogonal direct sum as follows:
\be\label{decomp}
TM_N = TN \oplus J\cdot TN,
\ene
where $J$ is the complex structure on $M^{2n}$ such that the Riemannian metric on $M^{2n}$ is given by $\omega(X,JY)$ for
tangent vectors $X$ and $Y$.

A Lagrangian submanifold $N^n$ is {\it minimal} if its {\mc} vector is identically zero. Similar to the {\ms} case, these minimal
submanifolds are critical points of the volume {\wrt} induced measure from $M^{2n}$ (\cite{Sim68}).

We now consider the space $\CH$, and use the projective model. Consider the inner product
$\langle\cdot,\cdot\rangle$ on $\mathbb C^{2,1}$ as follows:
\be\label{inner}
\langle v,w\rangle = v_1\bar w_1+v_2\bar w_2-v_3\bar w_3.
\ene
We denote the cone $W_-=\{v\in\mathbb C^{2,1} : \langle v,v\rangle<0\}$, and let $P$ be the natural projection
$\mathbb C^{2,1}\setminus\{0\}\to\mathbb{CP}^2$. Then we define $\CH$ as a complex manifold to be the image
of $W_-$ under this projection $P$. The space $\CH$ carries a natural tautological $S^1$-bundle: the pseudo-sphere
\beq
S_- =\{u \in W_-: \langle u,u\rangle = -1\},
\eeq
and $-\langle \cdot,\cdot\rangle$ induces a metric connection on the $S^1$-bundle: $\pi: S_- \to \CH$.

Note that $\CH$ does not admit any real {\tg} hypersurface, but
there are two kinds of {\tg} co-dimension two subspaces (see
\cite{Gol99}), namely, the complex line of constant curvature $-1$,
and the real Lagrangian plane $\RH$ of constant curvature
$-\frac14$. All {\mli}s we consider here are from the {\uc} of a
closed surface into $\CH$. Equation \eqref{e} has a trivial
solution, when $t = 0$, $u = 0$. This trivial solution corresponds
to the {\tg} Lagrangian embedding of $\RH$ in $\CH$, and the
corresponding representations of the surface group are Fuchsian. Any
solution $(u(t),t)$ obtained near $(0, 0)$ gives rise to complex
hyperbolic {\qf} representations, as seen in \cite{LM10}. The
geometry of representations  for solutions $(u(t),t)$ outside a
neighborhood of $(0, 0)$ is unclear.

%%%%%%%%%%%%%%%%%%%%%%%%%%%%%%%%%%%%%
\subsection{The Structure Equation and the Reconstruction} In this subsection, we briefly recall the derivation and setup of
the structure equations for {\mls}s in $\CH$.

Let $D =\tilde{\Sigma} = \{z\in \mathbb{C}: |z| < 1\}$ be the unit disk, the {\uc} of the surface $\Sigma$, then for
any Lagrangian immersion $\varphi: D \to \CH$, it admits a horizontal Legendrian lift $f: \tilde{\Sigma} \to S_-$. This lift
gives rise to a frame, for any $z \in D$:
\be\label{Frame}
F = (f_z/|f_z|\,\,\, f_{\bar z}/|f_{\bar z}| \,\,\, f).
\ene
This frame $F(z)$ lies in $U(2,1)$. Furthermore, it is shown in \cite{LM10} that $F$ lies in $SU(2,1)$ if and only if $\varphi$ is
a {\mli}.

Let $\varphi$ be a conformal Lagrangian immersion of the unit disk $D$ to $\CH$. There is a local Legendrian lift
$f\!: D \to S_-\subset \mathbb C^{2,1}$ so that
\beq
\langle f,f\rangle = -1,\qquad \langle f,f_z
\rangle = \langle f, f_{\bar z} \rangle = 0,\qquad \langle
f_z,f_{\bar z}\rangle = 0,
\eeq
and we write the first fundamental form $\langle f_z,f_z\rangle = \langle f_{\bar z},f_{\bar z} \rangle$ as $s^2$, then
$2s^2|dz|^2$ is the local expression of the metric on $D$. For a conformal map, the minimality of $\varphi$ is just the
condition for it to be a harmonic map
$$\langle f_{z\bar z}, f_z \rangle  = \langle f_{z\bar z} , f_{\bar z}
\rangle= 0.$$

Define
\beq
q = \langle f_{zzz},f\rangle = -\langle f_{zz},f_{\bar z}\rangle,
\eeq
and we may compute by taking $z$ and $\bar z$ derivatives
of the above equations to find:
$$ \langle f_{zz},f_z \rangle = 2ss_z, \qquad \langle f_{\bar z \bar
z} , f_{\bar z} \rangle = 2ss_{\bar z}, \qquad \langle f_{\bar z
\bar z}, f_z \rangle = \bar q.$$

The frame $F =(f_z/|f_z|\,\,\, f_{\bar z}/|f_{\bar z}| \,\,\, f)$ lies in $SU(2,1)$, and we define the Maurer-Cartan
form
\be\label{alpha}
\alpha = F^{-1}dF = A dz + Bd\bar{z},
\ene
where
\beq
A = F^{-1}F_z = \left(\begin{array}{ccc} (\log s)_z &
0 & s \\
-qs^{-2} &- (\log s)_z & 0 \\
0 & s & 0
\end{array} \right),
\eeq
and
\beq
 B =F^{-1}F_{\bar{z}} = \left(\begin{array}{ccc} -  (\log s)_{\bar z} &
\bar q s^{-2} & 0 \\
0 & (\log s)_{\bar z} & s \\
s& 0 & 0
\end{array} \right).
\eeq
The Maurer-Cartan equations $d\alpha+\alpha\wedge\alpha=0$ are equivalent to the following
\begin{equation} \label{e-local}
   \left\{
  \begin{array}{cl}
   q_{\bar z}=0    &    \\
     \frac{\partial^2} {\partial z
\partial \bar z} \log(s^2) = |q|^2 s^{-4} + s^2    &
  \end{array}. \right.
\end{equation}
Now using $\Delta$ as the Laplacian of the {\hym} $g_\sigma$ on $(\Sigma, \sigma(z))$, and
\beq
\Delta = \frac{4}{g_\sigma}\frac{\partial^2}{\partial z\partial\bar{z}}, \ \ \
-1 = -\frac{2}{g_\sigma}\frac{\partial^2}{\partial z\partial\bar{z}}\log(g_\sigma).
\eeq
Since $e^u g_\sigma = 2s^2$,  the second equation of \eqref{e-local} becomes the local version of
\be\label{structure}
\Delta u + 2 - 2e^u - 16\|q\|^2 e^{-2u} = 0.
\ene
Since the Maurer-Cartan equations are the integrability conditions for the frame $F$, we have the following
local characterization of minimal Lagrangian immersions in $\CH$:
\begin{pro}\label{reduction}
A conformal minimal Lagrangian immersion $\varphi\!:D\to \CH$ induces a
holomorphic cubic differential $q$ and metric $2s^2|dz|^2$ which
satisfy \eqref{e-local}. Conversely, if $\sigma$ is a {\cs} on $\Sigma$ and $q$ is a {\hcd} on
$(\Sigma,\sigma)$. Then any solution $u: \Sigma \to \R$ to \eqref{structure} determines a Legendrian
frame $F\!:D = \tilde{\Sigma}\to SU(2,1)$ for a {\mli} $\varphi: D \to\CH$ which is equivariant {\wrt}
some surface group representation from $\pi_1(\Sigma)$ into $SU(2,1)$. The immersed minimal
Lagrangian surface is unique up to holomorphic isometries of $\CH$.
\end{pro}

This Proposition provides the reconstruction scheme from solving
\eqref{structure}: Given a {\cs} $\sigma$ on the closed surface
$\Sigma$, and a {\hcd} $q$ on $(\Sigma, \sigma)$, each solution
$u(z)$ of equation \eqref{structure} corresponds to an induced
metric $2s^2 |dz|^2$ on $D$. From the formula in \eqref{alpha}, the
$1$-form $\alpha$ is therefore determined. As in \cite{LM10}, one
integrates the Maurer-Cartan equations $F^{-1}dF = \alpha$ to obtain
a frame $F = (f_z/|f_z|\,\,\, f_{\bar z}/|f_{\bar z}| \,\,\, f)$.
The Lagrangian and minimal properties are encoded in verifying $F
\in SU(2,1)$. Therefore, we reduce the problem of obtaining surface
group equivariant {\mli} of the disk $D$ into $\CH$ to the
solvability of equation \eqref{structure} of data $(\sigma, q)$.

%%%%%%%%%%%%%%%%%%%%%%%%%%%%%%%%%%%%%%%%
\subsection{The Second Fundamental Form}

For any immersion problem, the {\sff} is a natural key object of the study, and they often give rise to the concepts
of {\pc}s. In the case of minimal hypersuface in a {\htm}, the {\sff} is described as the real part of a {\hqd}. In this
subsection, we show that an analogous fact is true for minimal Lagrangian surfaces in $\CH$, namely, we describe
all the components of the {\sff} in terms of real and imaginary parts of a {\hcd} $q$.

To compute the second fundamental form, we consider for a conformal coordinate $z=x+iy$ the orthonormal
basis of the tangent space of the immersed surface given by $E_1=f_x/|f_x|$, $E_2 = f_y/|f_y|$. Since the
surface is Lagrangian, by \eqref{decomp}, $(E_1, E_2,iE_1,iE_2)$ form an orthonormal basis in the tangent
space of $\CH$. For tangent vector fields $X,Y\in \Gamma(T\varphi(D))$, the {\sff} of the immersion into
$\CH$ is given by $$ \mathrm{I\!I}(X,Y) = \sum_{j=1}^2 g(\nabla_X Y,iE_j)iE_j,$$ where $g(v,w)$ is the
Riemannian metric on $\CH$ inherited from the inner product \eqref{inner} on ${\mathbb C}^{2,1}$, and
$\nabla$ is the Levi-Civita connection on $\CH$, which is the projection of the flat connection on ${\mathbb C}^{2,1}$.

We find that there are three independent entries for the {\sff} $\II$, which allows us to arrange them in a
$2 \times 2$ symmetric matrix:

\begin{pro}\label{sff}
All the components of $\II$ are determined by the metric $2s^2|dz|^2$
and the cubic differential $q$. In particular, we have
\begin{eqnarray*}
 \II(E_1,E_1) &=& 2^{-\frac12}s^{-3}(-\,{\rm Im}\,q\cdot iE_1
 -{\rm Re}\,q\cdot iE_2),\\
   \II(E_1,E_2) &=& 2^{-\frac12}s^{-3}(-\,{\rm Re}\,q\cdot iE_1
   +{\rm Im}\,q\cdot iE_2), \\
 \II(E_2,E_2) &=& 2^{-\frac12}s^{-3}(\,{\rm Im}\,q\cdot iE_1
 + {\rm Re}\,q \cdot iE_2).
\end{eqnarray*}
\end{pro}
\begin{proof}
We compute only $\II(E_1,E_1)$, as the rest are similar. Since the Levi-Civita connection on $\CH$ is the projection of
the flat connection on ${\mathbb C}^{2,1}$, we can compute $\nabla_X Y$ as $(XY)f$. Therefore,
\begin{eqnarray*}
\II(E_1,E_1) &=& \frac1{2s^2} \II(f_x, f_x) \\
&=& \sum_{j=1}^2 \frac1{2s^2} g(f_{xx},iE_j)\cdot iE_j \\
&=&\frac1{2\sqrt2\cdot s^3} [g(f_{xx}, i f _x)\cdot iE_1 + g(f_{xx},
i f_y)\cdot i E_2] \\
\end{eqnarray*}
Now compute
\begin{eqnarray*}
g(f_{xx},if_x) &=& {\rm Re}\langle f_{zz} + 2f_{z\bar z} + f_{\bar z
\bar z}, if_z+if_{\bar z} \rangle \\
&=& {\rm Re}[-i(2ss_z-q+\bar q + 2ss_{\bar z})] \\
&=& -2\,{\rm Im}\, q.
\end{eqnarray*}
We may similarly compute $g(f_{xx},if_y) = -2\,{\rm Re}\,q$. So
altogether,
$$\II(E_1,E_1) = \frac1{2\sqrt2\cdot s^3} [-2\,{\rm Im}\, q \cdot
iE_1 - 2\,{\rm Re}\, q \cdot iE_2].$$
\end{proof}
\vskip 0.1in

\section{General Existence Results}

We now consider a family of equations \eqref{structure} determined the ray $tq$ ($t \geq 0$) in the space of
{\hcd}s $C(\sigma)$:
 \be\label{t-e}
\Delta u(z,t) + 2 -2e^{u(z,t)} -16t^2\|q\|^2 e^{-2u(z,t)} = 0.
 \ene
 Recall from \eqref{q} that $\|q\|^2 = \frac{q\bar{q}}{g_\sigma^{3}}$.

 It is an immediate consequence of the {\maxp} that we have:
\bpo\label{upbd}
Any solution $u$ to \eqref{t-e} satisfies  $u \leq 0$.
\epo

Since $C(\sigma)$ is a finite dimensional vector space, this
approach allows us to fix $q$ and focus on finding interval of the
parameter $t$ for which solutions to \eqref{t-e} exists. This set up
is quite standard in nonlinear analysis, where we have access to
several important techniques such as the continuity method and the
variational method. Equations \eqref{e} and {\eqref{t-e} are
very similar to the equations in the problem of minimal immersions
of closed surfaces into {\htm}s first introduced by Uhlenbeck
\cite{Uhl83} and further studied in \cite{HL11}.

%%%%%%%%%%%%%%%%%%%%%%%
\subsection{Nonexistence} In this subsection, we deal with (possibly) large values for parameter $t$.
\bt\label{non} There exists a constant $T = T(q, \sigma)$ such that
equation \eqref{t-e} does not admit any solution for any $t \geq
T$. \et Note that, the solvability of equation \eqref{t-e} is a
necessary condition for the existence of a {\mli} of $\Sigma$ into
$\CH$. This theorem is similar to Proposition 5.8 in
\cite{LM10}. We include a proof here for the sake of completeness
since it is very short. \bp We integrate  equation \eqref{t-e}
with respect to the {\hym} $g_\sigma$: \be \label{t} A_\sigma = 8t^2
\int_\Sigma \|q\|^2 e^{-2u}dA_\sigma + \int_\Sigma e^{u}dA_\sigma >
8t^2 \int_\Sigma \|q\|^2 e^{-2u}dA_\sigma. \ene Meanwhile, we apply
the H\"{o}lder's inequality, and Proposition \ref{upbd}: \bear
\int_\Sigma \|q\|^{2/3} dA_\sigma &=& \int_\Sigma \|q\|^{2/3} e^{-2u/3}e^{2u/3}dA_\sigma \\
& \leq & \{ \int_\Sigma \|q\|^2 e^{-2u}dA_\sigma\}^{1/3}\{\int_\Sigma e^{u}dA_\sigma\}^{2/3}\\
& \leq & A_\sigma^{2/3}\{\int_\Sigma \|q\|^2 e^{-2u}dA_\sigma\}^{1/3}.
\eear
Applying above inequality to \eqref{t}, and noting that the hyperbolic area of $\Sigma$ is $A_\sigma = 2\pi(2g-2)$, we find:
\beq
t < \{\frac{2\pi(g-1)}{\int_\Sigma \|q\|^{2/3} dA_\sigma}\}^{3/2}.
\eeq
Now we can simply choose $T = \{\frac{2\pi(g-1)}{\int_\Sigma \|q\|^{2/3} dA_\sigma}\}^{3/2}$.
\ep

%%%%%%%%%%%%%%%%%%%%%%%%%%%%
\subsection{Existence} In this subsection, we take advantage of this variational setup to apply the {\ift} to prove the
solvability of equation \eqref{t-e}, therefore the general existence of {\mli} of $\tilde\Sigma$ into $\CH$.

We consider the nonlinear map $F: W^{2,2} (\Sigma) \times [0,\infty) \to L^2 (\Sigma)$ defined by
\be\label{FF}
  F(u,t) = \Delta u + 2  - 2e^u - 16t^2\|q\|^2e^{-2u},
\ene
where $W^{2,k} (\Sigma)$ stands for the classical Sobolev space. At each $t \geq 0$ fixed, the {\lo}
$L(u,t) : W^{2,2} (\Sigma) \to L^2 (\Sigma)$ associated to $F$ is given by
\be\label{L}
  L(u,t) =- \Delta  + 2e^{-2u} \left(e^{3u} - 16t^2 \|q\|^2\right).
\ene
It is easy to see that $L(u,0) > 0$, since $-\Delta$ has nonnegative eigenvalues. In standard theory, the operator $L(u,t)$
in \eqref{L} is crucial in order to apply the implicit function theorem. In particular, when the {\lo} $L$ has all positive
eigenvalues, the differential of the map $F(u,t)$ in \eqref{FF} is onto.

The existence of solutions for small $t$ is implied by the following:
\bt \label{curve}
There exist a constant $T_0 = T_0(\sigma, q) > 0$ and a smooth curve
$$
   \gamma: [0,T_0] \to  W^{2,2} (\Sigma) \times [0,\infty)
   \qquad
   t \mapsto (u(t),t),
$$
such that
\begin{enumerate}
\item[(a)]
$\gamma(0) = (0,0)$ and $F(\gamma(t)) = 0$ for all $ t \in [0, T_0]$.
\item[(b)]
$L(u(t), t) > 0 $ for all $t \in [0, T_0)$.
\item[(c)]
${\rm Ker} \big( L(u(T_0), T_0) \big) \not = \{0 \}$.
\item[(d)]
The family of solutions to $F(u(t),t) = 0$ is unique near $\gamma(0)
= (0,0)$.
\end{enumerate}
\et
\bp
We follow closely the existence of solutions in
\cite{Uhl83}.
%Since any solution $u(z)$ satisfies $u \leq 0$,
%we have
% \be\label{tq}
%   e^{3u} - 2t^2 \|q\|^2 \leq \frac18 - 2t^2 \|q\|^2.
% \ene
%Therefore for large value of $t$, the first eigenvalue of the {\lo}
%$L(u,t)$ is negative at any possible solution.

We use the continuity method. Let $E=\{t\in[0,\infty) :$ there is a
unique smooth solution $\gamma(\tau)$ to $\gamma(0)=(0,0)$ so that
$F(\gamma(\tau))=0$ and $L(u(\tau),\tau)>0$ for all
$\tau\in[0,t]\}.$ Clearly $E$ includes 0. $E$ is open by the {\ift}
and since $L(u(t),t)>0$. Applying  Theorem \ref{u-bound} below
and standard elliptic theory, we find that $E$ is relatively closed
in $[0,T_0)$, where $T_0$ is defined as the smallest $t$ so that (c)
holds. Therefore the statements (a) (b) (c) and (d) follow.

\ep
\br
\ben
\item
Note that by standard regularity theory, the solution $u(t)$
obtained above belongs to the class $C^{\infty}(\Sigma)$.
\item One can show that for $t^2\|q\|^2\le\frac1{54}$, a lower bound
on $u$ along $\gamma$ from \cite{LM10} implies that $L$ is a
positive operator, and thus we find an alternate proof of the
existence result in \cite{LM10}.
%From \eqref{tq}, it is easy to see that the existence of solutions
%when the condition $t^2\|q\|^2 \leq \frac1{16}$ is satisfied.
%Previously in \cite{LM10}, the upper bound for $t^2\|q\|^2$ is given
%by $\frac{1}{54}$.
\een
\er
%%%%%%%%%%%%%%%%%%%%%%%%%%
\subsection{Estimates for closedness}
%In this subsection, we prove part (iii) of the Theorem \ref{main}, namely,
%\bt\label{ext}
%The {\scv} $\gamma$ in the Theorem \ref{curve} can be extended to $t = T_0$.
%\et
%\vskip 0.1in
We start with a lemma that we will use later:
\bl \label{leg-trans}
For $a\ge 1$ and $b\ge 0$, we have
\be \label{leg}
ab\le H(a)+H^*(b),
\ene
where $H(a) = \frac14 a(\log a)^2$ and $H^*(b)= \frac12 e^{-1+\sqrt{1+4b}}(-1+\sqrt{1+4b})$.
\el
\begin{proof}
For $a\ge 1$, $H(a)$ is a convex function. Then we may compute $$H^*(b) = \sup_{a\ge 1} [ab-H(a)]$$ as
the Legendre transform of $H$. Indeed, we have
\beq
H'(a) = \frac14 (\log a)^2 + \frac12 \log a,
\eeq
and consequently, for $b=H'(a)$,
\beq
H^*(b)=H^*(\frac14 (\log a)^2 + \frac12 \log a) = aH'(a) - H(a) = \frac12 a\log a.
\eeq
It is then easy to solve for the formula of $H^*(b)$.
\end{proof}
We are left to show the following key estimate to complete the proof of  Theorem \ref{curve}, namely, the
solutions on $t \in [0,T_0)$ given by the {\ift} can be extended to $t = T_0$.
\begin{rem}
The following type of estimate is known to Uhlenbeck \cite[p.\ 164]{Uhl83}, but
the proof is not included in \cite{Uhl83}.
\end{rem}
\bt\label{u-bound}
Let $p\in(1,\infty)$. Then there is a constant $C = C(\sigma, q, p)$ such that for every solution $u$ of along the path $\gamma$ in
 Theorem \ref{curve} from $t=0$ satisfying $F(u,t)=0$ and $L(u,t)\ge0$, we have $$ \|u\|_{W^{2,p}}\le C.$$
\et
\bp
By the {\ift} around $t=0$, we may assume there is a fixed $\epsilon>0$ so that $t \ge \epsilon$. Now we assume $u$ is a solution
on the {\scv} $\gamma$ with $L(u,t)\ge0$, and we integrate both sides of \eqref{t-e} {\wrt} the {\hym} on $\Sigma$ to find:
\beq
A_\sigma = \int_\Sigma e^u + 8t^2 \int_\Sigma \|q\|^2 e^{-2u},
\eeq
where we recall that $A_\sigma = 2\pi (2g-2)$ is the hyperbolic area of $\Sigma$.

Since  $t\ge\epsilon$, we have for a positive constant
$C_1 = C_1(\epsilon)$,
\be \label{bound-int-term}
 \int_\Sigma e^{-2u}\|q\|^2\le C_1.
\ene
Since $q$ is a prescribed {\hcd} on a closed surface $(\Sigma,\sigma)$, it is well-known that $q$
has isolated zeros.

To derive an integral bound on $u$, let $n$ be the largest order of
all the zeros of $q$, and let
$\ell<\frac1{n+1}$, $\alpha=\frac1\ell$ and $\beta$ be the conjugate exponent of $\alpha$ (namely $\frac1\alpha + \frac1\beta =1$).
Now $\|q^{-2\ell}\|_\beta$ is finite, which allows us to apply H\"older's Inequality:
\beq
\int_\Sigma (e^{-2u})^\ell \le \|(e^{-2u}\|q\|^2)^\ell\|_\alpha \cdot \|q^{-2\ell}\|_\beta<C_2,
\eeq
for some $C_2 = C_2(\sigma, q, \epsilon) > 0$.

Since $u\le0$ (Proposition \ref{upbd}), each $|u|^p$ is dominated by $e^{-2\ell u}$, and we have uniform $L^p$ bounds, for any $p > 1$,
\be \label{Lp-bound}
\|u\|_p \le C_3 = C_3 (\sigma, q, \epsilon, p).
\ene
\vskip 0.1in
Furthermore, since $\langle L(u,t)u,u\rangle \ge 0$, we have
\begin{eqnarray}\label{Lu}
\langle L(u,t)u,u\rangle &=& \int_\Sigma (-\Delta u + 2e^{-2u}(e^{3u} -16t^2\|q\|^2)u)u \nonumber \\
&=& \int_\Sigma |\nabla u|^2 + 2e^u u^2 - 32e^{-2u}t^2\|q\|^2 u^2 \ge 0.
\end{eqnarray}
\vskip 0.1in
We then multiply equation \eqref{t-e} by $u$ and integrate by parts to find:
 \be\label{uFu}
 \int_\Sigma  |\nabla u|^2 = \int_\Sigma 2u -2ue^u -16t^2\|q\|^2ue^{-2u} .
 \ene
  Applying \eqref{uFu} to the inequality \eqref{Lu}, we have
 \be\label{u-integral-ineq} \int_\Sigma 2u + 2e^u(u^2-u) -
16t^2\|q\|^2e^{-2u}(2u^2+u) \ge 0.
 \ene
 \vskip 0.1in
  Now the
combination of $u\le0$ and the $L^p$-bound \eqref{Lp-bound} gives the
following:
 \be \label{bd-integral} \int_\Sigma
16t^2\|q\|^2e^{-2u}u^2\le C_4,
 \ene
 for some uniform constant $C_4 =
C_4 (\sigma, q, \epsilon, p) > 0$. We note the extra $u^2$ in this
integral, together with the Green's function representation of
solutions to \eqref{t-e}, will be enough to prove uniform
$L^\infty$-estimates on $u$.

To proceed, we define the following simplified notion for two functions $f$ and $g$: $f = g + O(1)$ if $|f-g|$
is uniformly bounded from above by some positive constant.

Let $G(x,y)$ be the Green's function for the  hyperbolic Laplacian. We have
\beq
G(x,y) = \frac1{2\pi}\log (d(x,y)) + O(1),
\eeq
for $d$ the hyperbolic distance.

Let $\bar u$ be the average of $u$, which is bounded by the $L^p$-bound \eqref{Lp-bound}. We apply the
Green's formula to find:
\bear
u(x) &=& \bar u + \int_\Sigma G(x,y) \Delta u(y)\,dVol(y) \\
&=& \int_\Sigma G(x,y)[-2+2e^{u(y)}+16t^2\|q(y)\|^2e^{-2u(y)}]\,dVol(y) + O(1)\\
&=& 16t^2\int_\Sigma G(x,y)\|q(y)\|^2e^{-2u(y)}\,dVol(y) + O(1).
\eear

Now choose a complex normal coordinate disk $\mathcal D$ centered at
$x$ (so that $x=0$), and use the asymptotics of the Green's
function, together with (\ref{bound-int-term}), to find
 \beq u(0) =
\frac{8t^2}{\pi} \int_{\mathcal D} \log|y| e^{-2u} \|q\|^2 \,dVol(y)
+ O(1). \eeq
\vskip 0.1in
We want to compare this integral to $\int
t^2 e^{-2u}u^2\|q\|^2$, for which we have a bound by
\eqref{bd-integral}. We now choose $a=e^{-2u}$ and
$b=\big|\log|y|\big|$. It is easy to verify that $a \ge 1$ and
$b=\big|\log|y|\big| \ge 0$. Hence $H(a) = u^2 e^{-2u}$ and $H^*(b)=
\frac12 e^{-1+\sqrt{1+4b}}(-1+\sqrt{1+4b})$. We can now apply
Lemma \ref{leg-trans} to find: \bear
 |u(0)| &\le& \frac{8t^2}{\pi}\int_{\mathcal D} \big|\log|y|\big|\cdot e^{-2u}\cdot \|q\|^2\,dVol(y) + O(1)\\
 &\le&  \frac{8t^2}{\pi}\int_\Sigma e^{-2u}u^2\|q\|^2 +{ \frac{8t^2}{\pi}} \int_{\mathcal D} H^*(b)\|q\|^2\,dVol(y)  + O(1).
 \eear
 Both these terms are bounded, the first by (\ref{bd-integral}), and the second by a direct computation.
 Since $0=x\in\Sigma$ was arbitrary, we have a uniform bound $$\|u\|_{L^\infty}\le C_5$$ This bound can then be plugged into
 the equation $F(u,t)=0$ to find uniform $L^\infty$ bounds on $\Delta u$. Thus standard $L^p$ theory applies, and we have
 uniform $W^{2,p}$ bounds on $u$. Higher regularity is standard.
\ep
%Now the Theorem \ref{ext} follows from the uniform estimate in Theorem \ref{u-bound} immediately. Indeed, if $\{u_n\}$ is a
%sequence of solutions on the {\scv} $\gamma$, with $L(u_n) > 0$ and converges some $\hat{u}$ weakly as $t_n \to T_0$.
%Theorem \ref{u-bound} guarantees that this convergence is actually strong, with $\hat{u}$ also a solution to \eqref{t-e}.
\section{Non-Uniqueness}
In previous sections, we have proved parts (ii) and (iii) of  Theorem \ref{main}. By the {\ift}, the family of solutions to the
structure equation is unique for the family including $\gamma(0)$. In this section, we address the issue of nonuniqueness
for this problem, i.e., we construct a {\mpa} type solution for each parameter value $t$ on $(0,T_0)$. This will complete
the proof of part (i).

%%%%%%%%%%%%%%%%%%%%%%%%
\subsection{New Formulation} We start with a new formulation of the problem in order to prove a compactness
result. This is necessary because the original Euler-Lagrange functional associated to the structure equation \eqref{t-e} does not
satisfy a compactness property that is required to apply the {\mpt}. Our approach is to follow the strategy used in \cite{HL11} for
the {\mi} problem in {\htm}s: we define a new functional and a new norm for the structure equation \eqref{t-e}, and show
the critical points of the new functional coincide with the solutions of \eqref{t-e}, and in next subsection we prove a compactness
theorem for the new functional and norm, and finally apply the {\mpt} in \cite{AR73} to produce a second solution for each $t$ on $(0,T_0)$.

To proceed, we need to capture the nonlinearities arising from the structure equation \eqref{t-e}, which now we recall:
\beq
\Delta u + 2 - 2e^{u} -V e^{-2u} = 0,
\eeq
where we set $V =V(t,z) = 16t^2\|q\|^2$.

Let $H^1(\Sigma)$ be the usual Sobolev space
$$
   H^1 (\Sigma) := \{ u \in L^2 (\Sigma) \, \colon \, \nabla u \in L^2 (\Sigma) \},
$$
equiped with the norm
$$
   \langle f, g\rangle_{H^1} := \int_{\Sigma} \Big\{ \nabla f \nabla g + fg \Big\}.
$$
\vskip 0.1in
Then for $u \in H^1(\Sigma)$, the Euler-Lagrange functional for \eqref{t-e} is
\be\label{I}
  I(u) :=
  {\frac12}\int_{\Sigma}|\nabla u|^2 + \int_{\Sigma} \left(2e^u - 2u -{V\frac{e^{-2u}}{2}}\right).
\ene
Note that this functional does not satisfy the Palais-Smale compactness condition.
\vskip 0.1in
We now explicitly construct smooth functions as follows:
\be\label{f-1}
  f_1 (s) := \left\{
  \begin{array}{cl}
    2 - 2e^{s}                     & \hbox{ if } s \leq 0 \\
    -\theta s^{\theta -1}          & \hbox{ if } s > 1
  \end{array}, \right.
\ene where $\theta > 2$ is a constant and we require $f_1(s)<0$ for
$s>0$. Let \be\label{f-2}
  f_2 (s) := \left\{
  \begin{array}{cl}
    s - e^{-2 s}        & \hbox{ if } s \leq 0 \\
    0                   & \hbox{ if } s > 1
  \end{array}. \right.
\ene
In addition, it is easy to see that we can also require $f_2(s) < 0$ for all $s \in (0,1)$. \vskip 0.1in With these functions, we can
transfer equation \eqref{t-e} to a new equation which is better
suited for variational methods: \bl\label{NG} The
structure equation ~\eqref{t-e} is equivalent to the new equation
\be\label{new}
  - \Delta u + V u - \big( f_1 (u) +  V f_2 (u) \big) = 0  \,.
\ene
\el
\begin{proof}
First, let $u$ be a solution to equation \eqref{t-e}. Then from Proposition \ref{upbd}, we have $u \le 0$. In this case, from the
explicit formulas in \eqref{f-1} and \eqref{f-2}, we find that $f_1(u) = 2-2e^u$ and $f_2(u) = u-e^{-2u}$.

We now verify that
\bear
-\Delta u + Vu -  f_1 (u) - V f_2 (u) &=& - \Delta u +Vu -2 +2e^{u} - V(z)u + Ve^{-2u} \\
&=& - \Delta u -2 + 2e^{u} + Ve^{-2u} \\
&=& 0.
\eear
Conversely, let $u$ be a solution to equation \eqref{new}. We now apply the {\maxp} to ~\eqref{new}.
At the maximum point $p_0$ of $u$, we have $\Delta u(p_0) \leq 0$. Note that these functions $f_1$ and $f_2$ enjoy the
following properties:
\beq
  f_1 (s) < 0 \quad \forall s > 0,
  \qquad
  f_2 (s) \leq min\{0,s\}
  \quad \forall s \in \mathbb R .
\eeq Therefore at $p_0$, we have either $u(p_0)\le0$ or
  \beq
 Vu(p_0) \leq f_1 (u(p_0)) + V(z) f_2 (u(p_0)) \leq  f_1 (u(p_0)) + Vu(p_0).
\eeq
This implies that $f_1(u(p_0)) \geq 0$, and therefore $u(p_0) \leq 0$ by the definition  \eqref{f-1} of $f_1$. Since
$p_0$ is the maximum of $u$, we have just showed $u \le 0$. It is then easy to see that the sets of solutions of
~\eqref{t-e} and ~\eqref{new} coincide.
\end{proof}
To define the functional for this new equation \eqref{new}, we need to introduce an appropriate norm on
$H^1(\Sigma)$ as follows:
\be\label{V}
  \langle f, g \rangle_V :=
  \int_{\Sigma} \Big\{ \nabla f \nabla g + V(z) fg \Big\}.
\ene
\vskip 0.1in
\bl \label{norms}
The norms $\| \cdot \|_{H^1}$, $\| \cdot \|_V$ on the Sobolev space $H^1(\Sigma)$ are equivalent.
\el
\bp
One direction is obvious, namely, since $V \in L^{\infty} (\Sigma)$, we have $\| \cdot \|_V \leq C \| \cdot \|_{H^1}$. To work
in the opposite direction, we set $\bar u:= \frac{1}{A_\sigma} \int_{\Sigma} u$ as the average of $u$ on $\Sigma$, where
$A_\sigma = 4\pi (g-1)$ is the hyperbolic area of the surface.

We start with
\be\label{triangle}
  \| u \|_{L^2} \leq \| u - \bar u \|_{L^2} + \| \bar u \|_{L^2},
\ene
and the first term on the right-hand side of \eqref{triangle} is taken care of by the Poincar\'e inequality, namely, we have
\be \label{eq:Poincare}
  \| u - \bar u \|_{L^2} \leq  C \| \nabla u \|_{L^2}.
\ene
\vskip 0.1in
For the second term, we use
\beq
|\bar u|^2 \leq 2u^2 + 2|u-\bar u|^2,
\eeq
and apply Poincar\'e inequality again to find:
\bear
  |\bar u|^2 \int_\Sigma V(z)  &\leq& 2 \int_\Sigma \Big\{ V(z) |\bar u - u|^2 + V(z) u^2 \Big\}
  \\
  &\leq&
  C \int_\Sigma \Big\{ |\bar u - u|^2 + V(z) u^2 \Big\}
  \\
  &\leq&
  C \int_\Sigma \Big\{ | \nabla u |^2 + V(z) u^2 \Big\} \\
  & = & C\|u\|_V^2.
\eear
Since $\int_\Sigma V >0$, we also bound the second term in the right-hand side of \eqref{triangle}, and complete the proof.
\ep
\vskip 0.1in
We integrate from the formulas \eqref{f-1} and \eqref{f-2} to define new functions:
\be\label{F-1}
  F_1 (s) := \left\{
  \begin{array}{cl}
    2s - 2e^{s} +2         & \hbox{ if } s \leq 0 \\
    -s^{\theta}          & \hbox{ if } s > 1
  \end{array}, \right.
\ene
and
\be\label{F-2}
  F_2 (s) := \left\{
  \begin{array}{cl}
    \frac12 (s^2 + e^{-2 s})        & \hbox{ if } s \leq 0 \\
    0                   & \hbox{ if } s > 1
  \end{array}. \right.
\ene
\vskip 0.1in
Note that the additive constant $2$ in the formula of $F_1(s)$ when $s \leq  0$ is designed such that $f_1(s) = F_1'(s) < 0$
when $s > 0$. Using these functions, we can now define the functional corresponding to the formulation in equation \eqref{new},
on the Hilbert space $H^1(\Sigma)$, as the following:
\be\label{F}
    \Fcal (u) :=  \frac{1}{2}
     \int_\Sigma \big\{  |\nabla u|^2 + V(z) u^2 \big\}
   - \int_\Sigma \big\{ F_1 (u) + V(z) F_2 (u)   \big\} ,
   \quad
   u \in H^1 (\Sigma) ,
\ene
\vskip 0.1in
This functional is well-defined by the Moser-Trudinger inequality, and it is clear that it is continuously differentiable.
\br
From this definition \eqref{F}, the critical points of $\Fcal$ are weak solutions of ~\eqref{new}, and hence solutions to the
structure equation \eqref{t-e} by Lemma ~\ref{NG}. Also by Lemma \ref{NG}, for each $t \in (0,T_0)$,
Theorem \ref{curve} provides a critical point of $\Fcal$ which is stable. Our next subsection is to show for each stable critical
point of $\Fcal$, there is another solution (of mountain pass type) corresponding to the same $t \in (0,T_0)$.
\er
\vskip 0.1in
We end this subsection by the following easy observation:
\bpo
Let $k$ be a negative constant, then we have
\be\label{neg}
\lim_{k \to -\infty}\Fcal(k) = -\infty.
\ene
\epo
%%%%%%%%%%%%%%%%%%%%%%%%%%%%%%%
\subsection{Strong Non-uniqueness} One more step remains before we can prove the (strong) non-uniqueness for the
{\mli} with $(\sigma, tq)$, where $\sigma \in \Tcal_g(\Sigma)$ and $q \in C(\sigma)$, for each $t \in (0,T_0)$, namely,
we have to show a compactness property for the functional $\Fcal$:
\bt\label{PS}
The functional $\Fcal$ in ~\eqref{F} satisfies the Palais-Smale compactness condition.
\et
\bp
Let us facilitate with the following notation. Let $O(1)$ as before, and we call a quantity $|f| = o(1)$ if $|f|$ tends to zero
when an appropriate limit is taken. Using this notation, the Palais-Smale compactness condition for the functional $\Fcal$
is equivalent to showing that any sequence of functions $\{u_n \in H^1(\Sigma)\}$ which satisfies
\be\label{Oo}
|\Fcal(u_n)| = O(1),   \qquad   \|\Fcal'( u_n) \|_{H^{-1}} = o(1),
\ene
admits a subsequence which converges strongly in $H^1(\Sigma)$.

Suppose $\{u_n \in H^1(\Sigma)\}$ is a sequence which satisfies \eqref{Oo}, and we will prove the
theorem in two steps: first we show $\{u_n\}$ is bounded, hence there is a weak limit in $H^1(\Sigma)$,
then we show this weak limit is actually strong.
\vskip 0.1in
Step one: $\|u_n\|_V = O(1)$.
\vskip 0.1in
To see this, we deduce from the expressions of the functional $\Fcal$ in \eqref{F}, the norm $\|\cdot\|_V$ in \eqref{V}, and
the assumption that $\Fcal(u_n) = O(1)$, we have
\be\label{un-V}
\frac12 \|u_n\|_V^2 \leq \int_\Sigma \{F_1(u_n) + VF_2(u_n)\} + O(1).
\ene
\vskip 0.1in
It is not hard to verify, from the definitions of functions $f_j(s)$ and $F_j(s)$, $j = 1,2$, by considering all three subintervals
for $s \in \R$: $(-\infty, 0)$, $[0, 1]$ and $(1,\infty)$, that, for $\theta > 2$,
\be\label{fF}
F_j(s) \leq \frac{s}{\theta} f_j(s) + O(1), \qquad j =1,2.
\ene
\vskip 0.1in
On the other hand, in the direction of $\xi$, we have
\beq
\Fcal'(u)(\xi) = \int_\Sigma (\nabla u\nabla \xi + Vu\xi) - \int_\Sigma \xi(f_1(u) + Vf_2(u)).
\eeq
We deduce from the assumption $\|\Fcal'( u_n) \|_{H^{-1}} = o(1)$ that
\beq
\Fcal(u_n)(u_n) = o(1)\|u_n\|_V.
\eeq
Therefore we have
\be\label{un-V2}
\|u_n\|_V^2 = \int_\Sigma u_n\{f_1(u_n) + Vf_2(u_n)\} + o(1)\|u_n\|_V.
\ene
\vskip 0.1in
We now continue from the estimate \eqref{un-V} to find
\bear
\frac12 \|u_n\|_V^2 &\leq& \int_\Sigma \frac{u_n}{\theta}\{f_1(u_n) + Vf_2(u_n)\} + O(1)\\
&=& \frac{1}{\theta} \|u_n\|_V^2 + O(1) + o(1)\|u_n\|_V.
\eear
Step one is now completed since $\theta>2$ is a constant. Therefore, we obtain, up to a subsequence, $\{u_n\}$ converges
weakly in $H^1(\Sigma)$ to some $\hat{u}$.
\vskip 0.1in
Step two: $\{u_n\}$ converges strongly to $\hat{u}$, namely, $\|u_n - \hat{u}\|_{H^1} = o(1)$.
\vskip 0.1in
To complete our proof, by the equivalence of two norms (Lemma \ref{norms}), we only have to show
$\|u_n - \hat{u}\|_{V}^2 = o(1)$. Meanwhile, we have
\be\label{hat}
\|u_n - \hat{u}\|_{V}^2 = \int_\Sigma \{(f_1(u_n)+Vf_2(u_n))(u_n-\hat{u})\}+o(1).
\ene

From the expression of $f_1(s)$ in \eqref{f-1} and the fact that
$H^1$ is compactly included in $L^p$ for all $p<\infty$ shows that
(perhaps going to a further subsequence)
%and Lebesgue dominated convergence theorem,
we have $\{f_1(u_n)\}$ converges strongly to $f_1(\hat{u})$ in
$L^2(\Sigma)$. Similarly, from the expression of $f_2(s)$ in
\eqref{f-2}, and the Moser-Trudinger inequality, we find that
$\{f_2(u_n)\}$ converges strongly to $f_2(\hat{u})$ in
$L^2(\Sigma)$. Now step two is completed from \eqref{hat}. \ep We
now prove our main theorem of the section.
 \bt \label{exist-2-sol}
 For each $t \in
(0,T_0)$, where $T_0$ is defined in  Theorem \ref{curve}, for
fixed {\cs} $\sigma \in \Tcal_g(\Sigma)$, and {\hcd} $qdz^3 \in
C(\sigma)$, the structure equation \eqref{t-e} admits at least two
solutions.
 \et \bp
Since  equations \eqref{t-e} and \eqref{new} are equivalent, according to
Lemma \ref{NG}, and $\Fcal$ is the associated functional to equation \eqref{new},
we only have to show $\Fcal$ admits at least two critical points for each $t \in (0,T_0)$.

From Theorem \ref{curve}, we have $\Fcal$ admits one critical point $(u(t),t)$ such that the {\lo} is positive.
Therefore this (stable) solution $u(t)$ obtained in  Theorem \ref{curve} is a local minimizer for the functional
$\Fcal$ in the Hilbert space $H^1(\Sigma)$. There then exists a ball $B(u(t),r)$ in $H^1(\Sigma)$ such that
\beq
 \inf_{v \in \partial B(u(t) , r) } \Fcal(v) \geq \Fcal(u(t) )  \,.
\eeq
However, the limit in \eqref{neg} indicates that there must be some function $w \in H^1(\Sigma)$ such that
\beq
   w \not \in B(u(t), r),
   \qquad
   \Fcal(w) < \Fcal(u(t)) \,.
\eeq
Since the Palais-Smale condition is satisfied by  Theorem \ref{PS}, the additional critical point of $\Fcal$ for any
$t \in (0,T_0)$ is obtained by applying the Mountain Pass Theorem of Ambrosetti-Rabinowitz \cite{AR73}.
\ep
Naturally, one is interested in these solutions when $t$ goes to zero.
\bt
Let $\{u_n(t_n)\}$ be a sequence of solutions to the structure equation \eqref{t-e}. If $t_n \to 0$ as $n \to \infty$, then
along a subsequence, we have
\ben
\item
$u_n \to 0$ uniformly; or
\item
$\|u_n\|_{\infty} \to \infty$. \een \et \bp From  equation
\eqref{t-e}, we have
 \be\label{-delta}
 -\Delta u_n = 2 - 2e^{u_n} -16t^2\|q\|^2 e^{-2u_n}. \ene
 As $t_n \to 0$, possibly up to a
subsequence, we have either $\|u_n\|_{H^1} = O(1)$ or $\|u_n\|_{H^1}
\to \infty$. Let us assume that $\|u_n\|_{H^1} = O(1)$ for the
moment. In this case, $u_n$ converges weakly to some $\hat{u}$. For
any $\xi \in C^{\infty}(\Sigma)$, and using \eqref{-delta}, we find:
\bear \int_{\Sigma}\nabla\hat{u}\nabla\xi \leftarrow
\int_{\Sigma}\nabla u_n\nabla\xi &=&
\int_{\Sigma}2\xi \{1-e^{u_n} - 8t^2\|q\|^2 e^{-2u_n}\}\\
&\to& \int_{\Sigma}2\xi \{1-e^{\hat{u}}\},
\eear
after taking $n \to \infty$. Since $\xi \in C^{\infty}(\Sigma)$ is arbitrary, we have:
\beq
-\Delta \hat{u} = 2(1-e^{\hat{u}}).
\eeq
It is then easy to see via the {\maxp} that $\hat{u} = 0$. By the uniqueness of the solution near
$\gamma(0) = (0,0)$ in  Theorem \ref{curve}, $\{u_n\}$ coincide with the solutions on $\gamma$ as $t_n \to 0$.

Now we consider the other possibility, namely, $\|u_n\|_{H^1} \to \infty$. We need to show $\|u_n\|_{\infty} \to \infty$. If
otherwise, assuming that $\|u_n\|_{\infty} = O(1)$, we integrate from \eqref{-delta} to obtain
\beq
\int_{\Sigma}|\nabla u_n|^2 = \int_{\Sigma} u_n \{2-2e^{u_n} - V(z)e^{-2u_n}\} = O(1).
\eeq
This contradicts the assumption that $\|u_n\|_{H^1} \to \infty$.
\ep

\begin{rem}
It is not in general clear whether the mountain-pass solutions we
produce in Theorem \ref{exist-2-sol} form a continuous family. Indeed, both the openness and closedness estimates along
$\gamma$ depend on the stability condition $L>0$, which we expect to
fail for the mountain-pass solutions. We can say more near $T_0$, as
the continuous family $\gamma$ bifurcates there (one can use the
implicit function theorem as in Uhlenbeck \cite[p.\ 157]{Uhl83}). It
is unclear whether the bifurcated solutions for $t=T_0-\epsilon$
coincide with the mountain-pass solutions we construct.
\end{rem}

%%%%%%%%%%%%%%%%%%%%%%%%%%%%%%%%
\section{The {\WP} Pairing}
In this section, we use the uniqueness of the solution of $F(u(t),t)$ on the solution curve $\gamma$ near
$\gamma(0) =(0,0)$ to define a functional on a
subspace of the space of {\mli}s in $\CH$. This functional turns out to have positive definite second variation over
{\TS}: a scalar multiple of the {\WP} pairing of {\hcd}s. In the case of closed {\ms} in a class of {\qf} manifolds, a
similar functional is shown (\cite{GHW10}) to be a potential function for the classical {\WPm} on {\TS}. This is also
an analog of the fact that the second variation of the energy functional for harmonic maps between closed surfaces
yields the {\WPm} on {\TS} (\cite{Wlf89}). We recall that the {\WP} pairing of {\hcd}s $q_1 dz^3$ and $q_2 dz^3$
is defined as
\be
\langle q_1, q_2\rangle_{WP} = \int_\Sigma \frac{q_1 {\bar q}_2}{g_\sigma^3}dA_\sigma.
\ene

Let $MLI(\Sigma)$ be the space of {\mli}s of (the covering of) a closed surface $\Sigma$ into $\CH$ and
$MLI_\epsilon(\Sigma)$ be the subspace of $MLI(\Sigma)$ such that the solution on $\gamma$ is unique
($\epsilon$- close to $\gamma(0) =(0,0)$). We have the complex dimension of $MLI_\epsilon(\Sigma)$ is
equal to $8g-8$. Now the following functional is well-defined:
\be
\Acal: MLI_\epsilon(\Sigma) \to \R, \ \ \ (\sigma, q) \mapsto -\int_\Sigma e^{u(t)}dA_\sigma,
\ene
i.e., it maps a {\mli} to negative of the surface area associated to the metric $e^{u}g_\sigma$. We consider this family
of functions $\Acal(t) = \Acal(\gamma(t))$ and its variations.

\bt We have the following
\ben
\item $\Acal(0) = 4\pi(1-g)$;
\item $\dot{\Acal} = \ddl{\Acal}{t}|_{t=0} = 0$;
\item $\ddot{\Acal} = \frac{d^2 \Acal}{dt^2}|_{t=0} = 16\int_\Sigma \|q\|^2 dA_\sigma$.
\een
\et
\bp
(i) Since $\gamma(0) =(0,0)$, we have $\Acal(0) = - \int_\Sigma dA_\sigma = 4\pi(1-g)$.
\vskip 0.1in

(ii) We denote $\dot{u} = \ddl{u}{t}|_{t=0}$, and differentiate equation \eqref{t-e} with respect to $t$ to find:
 \be\label{dotu}
 \Delta\dot{u} -32t\|q\|^2 e^{-2u} + 32t^2 \|q\|^2 e^{-2u}\dot{u} -2e^{u}\dot{u} = 0.
 \ene
 Now take value at $\gamma(0)$, we have
 \beq
 (\Delta -2)\dot{u} = 0,
 \eeq
 so $\dot{u} = 0$ from the {\maxp}. The claim (ii) now follows immediately.

 (iii) We denote the operator $D = -2(\Delta -2)^{-1}$. This is a positive, self-adjoint (with respect to the $L^2$ inner product
 of functions $\langle f_1,f_2 \rangle = \int_\Sigma f_1 f_2 dA_\sigma$) operator and $D(1) = 1$. Note that this
 operator plays a fundamental role in the {\WP} geometry of {\TS} (\cite{Wlp86}). Now we differentiate  equation
 \eqref{dotu} with respect to $t$ to find:
 \beq
  \Delta\ddot{u} -32\|q\|^2 e^{-2u} + 32t^2\|q\|^2 e^{-2u}\ddot{u} -2e^{u}\dot{u}^2 -2e^{u}\ddot{u} = 0.
 \eeq
 We evaluate above at $(0,0)$ to obtain
 \be
 (\Delta -2)\ddot{u} = 32\|q\|^2,
 \ene
and hence
 \be\label{ddot}
\ddot{u} = -16D(\|q\|^2).
 \ene
 Now the claim (iii) follows from the self-adjointness of the operator $D$ and \eqref{ddot}.
\ep

%%%%%%%%%%%%%%%%%%%%%%%
\bibliography{mlag16.bbl}
%\bibliographystyle{amsalpha}
%\bibliography{ref-Lag}
\end{document}